\documentclass[12pt]{amsart}
\usepackage{amsmath, amsthm, amscd, amssymb, amsfonts, latexsym, mathtools}
\usepackage{fullpage}
\usepackage{bm,mathdots}
\usepackage{dsfont}
\usepackage[all]{xy}
\usepackage{tikz}
\usepackage{float}
\usepackage{fancybox}
\usepackage{ytableau}
\usepackage{tikz,forest}
\usetikzlibrary{patterns}
\usetikzlibrary{arrows.meta}
\useforestlibrary{edges}
\usepackage[colorlinks=true, urlcolor=blue]{hyperref}
\usepackage{url}
\usepackage{longtable}
\usepackage{array}
\usepackage{xcolor}
\usepackage{todonotes}

\makeatletter
\renewcommand{\listoftodos}[1][\@todonotes@todolistname]{\@starttoc{tdo}{#1}}
\makeatother
\usepackage{cleveref}

\newcounter{rownum}

\usepackage{booktabs}
\usepackage{caption}
\usepackage{multicol}
\usepackage{multirow}
\numberwithin{equation}{section}

\newenvironment{red}{\relax\color{red}}{\relax}
\newenvironment{blue}{\relax\color{blue}}{\hspace*{.5ex}\relax}
\newenvironment{jaune}{\relax\color{magenta}}{\relax}
\newcommand{\ber}{\begin{red}}
\newcommand{\er}{\end{red}}
\newcommand{\beb}{\begin{blue}}
\newcommand{\eb}{\end{blue}}
\newcommand{\bej}{\begin{jaune}}
\newcommand{\ej}{\end{jaune}}

\newcommand{\kommentar}[1]{}

\newcommand{\F}{\mathbb F}

\newcommand{\Z}{\mathbb Z}
\newcommand{\Q}{\mathbb Q}

\renewcommand{\pmod}[1]{\,(\mathrm{mod}\,#1)}

\newtheorem{lem}{Lemma}[section]

\newtheorem{thm}[lem]{Theorem}

\newtheorem{cor}[lem]{Corollary}

\theoremstyle{definition}
\newtheorem{example}[lem]{Example}

\newtheorem{rem}[lem]{Remark}

\begin{document}

\title{Decision trees, Frobenius traces, and Weierstrass coefficients of elliptic curves}

\date{\today}

\author[B.S. Banwait]{Barinder S. Banwait}
\address{London, UK}
\email{\href{mailto:barinder.s.banwait@gmail.com}{barinder.s.banwait@gmail.com}}

\author[X. Huang]{Xiaoyu Huang$^{\dagger}$}
\thanks{$^{\dagger}$X.H.\ was supported by NSF grant DMS-2617281}
\address{Department of Mathematics, Temple University, Philadelphia, PA, 19122, USA}
\email{\href{mailto:xiaoyu.huang@temple.edu}{xiaoyu.huang@temple.edu}}

\author[K.-H. Lee]{Kyu-Hwan Lee$^{\star}$}
\thanks{$^{\star}$K.-H.L.\ was supported by NSF grant DMS-2617278.}
\address{Department of Mathematics, University of Connecticut, Storrs, CT 06269, USA  \hfill \break \indent Korea Institute for Advanced      Study, Seoul 02455, Republic of Korea}
\email{\href{mailto:khlee@math.uconn.edu}{khlee@math.uconn.edu}}

\author[S. Lee]{Seewoo Lee}
\address{Institute of Mathematics, \'{E}cole Polytechnique F\'{e}d\'{e}rale de Lausanne, CH-1015 Lausanne, Switzerland}
\email{\href{mailto:seewoo.lee@epfl.ch}{seewoo.lee@epfl.ch}}

\author[T. Oliver]{Thomas Oliver}
\address{University of Westminster, London, UK}
\email{\href{mailto:T.Oliver@westminster.ac.uk}{T.Oliver@westminster.ac.uk}}

\author[A. Pozdnyakov]{Alexey Pozdnyakov}
\address{Department of Mathematics, Princeton University, Princeton, NJ, 08544-1000, USA}
\email{\href{mailto:ap5763@princeton.edu}{ap5763@princeton.edu}}

\begin{abstract}
We investigate the extent to which the coefficients $(w_1,w_2,w_3,w_4,w_6)$ of the reduced minimal Weierstrass model of an elliptic curve $E/\mathbb{Q}$ are determined by the Dirichlet coefficients $a_n(E)$ of its $L$-function, whose values at primes of good reduction are the Frobenius traces of $E$.
We prove that $w_1$, $w_2$ and $w_3$ are given by explicit formulae in $a_2(E)$, $a_3(E)$ and $a_4(E)$, that $w_4$ modulo $5$ is then determined by $a_5(E)$, and that $w_6$ modulo $7$ is determined by $a_7(E)$ together with $w_1,w_2,w_3,w_4$.
These formulae, which appear to be new, were discovered by training decision tree models on the LMFDB; we report the accompanying experiments and explore applications to computing tables of elliptic curves.
\end{abstract}

\maketitle

% Outstanding items; remove this line (and \usepackage{todonotes}) before submission.
%\listoftodos[Outstanding items]

\section{Introduction}

Let $E/\Q$ be an elliptic curve given by a Weierstrass model
\begin{equation}\label{eq.Weierstrass}
y^2+w_1xy+w_3y=x^3+w_2x^2+w_4x+w_6.
\end{equation}
We write the coefficients as $(w_1,w_2,w_3,w_4,w_6)$ rather than the conventional $(a_1,a_2,a_3,a_4,a_6)$ to avoid a clash with the Dirichlet coefficients $a_n(E)$ of the $L$-function
\[L(E,s)=\prod_{p\, \nmid\, N(E)}\left(1-a_p(E)p^{-s}+p^{1-2s}\right)^{-1}\prod_{p\, \mid\, N(E)}\left(1-a_p(E)p^{-s}\right)^{-1}=\sum_{n=1}^{\infty}a_n(E)n^{-s},\]
where $N(E)$ denotes the conductor.
For $p\nmid N(E)$ the coefficient $a_p(E)$ is the Frobenius trace, and for $p\mid N(E)$ it is determined by the reduction type at $p$; see \Cref{s.background}.
Equivalently, $a_n(E)$ is the $n$-th Fourier coefficient of the weight $2$ newform attached to $E$.

Among all Weierstrass models for $E$, those with $|\Delta|$ minimal are called minimal, where $\Delta$ is the discriminant \cite[III.1]{silverman}.
Over $\Q$ there is moreover a unique \emph{reduced minimal model}, singled out by the normalisation
\begin{equation}\label{eq.reducedW}
w_1,w_3\in\{0,1\}, \qquad w_2\in\{-1,0,1\},
\end{equation}
see, for example, \cite[\S3.1]{cremona1997algorithms}.
The coefficients $w_1,w_2,w_3$ of the reduced minimal model thus take only finitely many values, fixed by what looks like an arbitrary bookkeeping convention.
Our main result is that they, together with $w_4$ and $w_6$ modulo $5$ and $7$, respectively, are recovered from the first few $a_n(E)$ through strikingly simple formulae.
To state these compactly, we write $[m]_n\in\{0,1,\dots,n-1\}$ for the residue of $m$ modulo $n$.

\begin{thm}\label{thm.main}
Let $E/\Q$ be an elliptic curve with reduced minimal Weierstrass coefficients $w_i$ $(i=1,2,3,4,6)$, and put $b_2=w_1^2+4w_2$ and $\varepsilon=\left[a_4(E)(a_4(E)-1)/2\right]_2$. Then
\begin{align}
w_1&=[a_2(E)]_2,\label{eq.thmw1}\\
w_2&=[a_3(E)+1-w_1]_3-1,\label{eq.thmw2}\\
w_3&=(1-w_1)\,\varepsilon+w_1\left[\varepsilon+1+w_2+(a_2(E)+1)/2\right]_2,\label{eq.thmw3}\\
w_4&\equiv3\left(a_5(E)-b_2^2-w_1w_3\right) \bmod 5,\label{eq.thmw4}\\
w_6&\equiv5\left(a_7(E)-b_2^3+b_2w_1w_3+w_3^2\right)+3b_2w_4 \bmod 7.\label{eq.thmw6}
\end{align}
Moreover $\varepsilon=[N(E)]_2$.
\end{thm}

The five formulae are to be read in sequence, each involving only Dirichlet coefficients and the $w_i$ determined by its predecessors; in particular the term $(a_2(E)+1)/2$ in \eqref{eq.thmw3} occurs only when $w_1=1$, in which case $a_2(E)=\pm1$ by the Hasse bound.
Thus $w_1,w_2,w_3$ are determined by $a_2(E),a_3(E),a_4(E)$ alone, and $w_4\bmod 5$ by these together with $a_5(E)$; the expression for $w_6$ requires in addition the integer $w_4$ itself, and not merely its residue modulo $5$.
Since the $a_n(E)$ depend only on the isogeny class of $E$, equations \eqref{eq.thmw1}--\eqref{eq.thmw3} recover the fact that $w_1,w_2,w_3$ are isogeny invariants \cite{siksek}.
The isogenous curves \texttt{11\!.\!a1} and \texttt{11\!.\!a2} show that $w_4$ and $w_6$ are not isogeny invariants.

Our choice of moduli in \eqref{eq.thmw4} and \eqref{eq.thmw6} follows \cite[Theorem~1.1]{siksek}: if $E$ is isogenous to $E'$, then $w_4'=w_4-5t$ and $w_6'=w_6-b_2t-7u$ for some $t,u\in\Z$.
Hence $w_4\bmod 5$ is again an isogeny invariant, whereas $w_6\bmod 7$ is not, and the term $3b_2w_4$ in \eqref{eq.thmw6} is precisely what accounts for the discrepancy.
Finally, $\varepsilon$ records whether $E$ has good reduction at $2$, and \eqref{eq.thmw3} is most naturally proved in the equivalent form obtained by replacing $\varepsilon$ with $[N(E)]_2$; writing it through $a_4(E)$ keeps the whole theorem inside the Dirichlet coefficients.

These formulae were first observed by applying decision trees to the elliptic curve database in the $L$-functions and Modular Forms Database (LMFDB) \cite{lmfdb}.
Those for $w_1$ and $w_2$ can be read off the relevant trees and verified by brute force computation; that for $w_3$ was suggested by a large language model prompted with the shape of the first two, and then verified in the same way (see Remarks~\ref{rem:w2chatgpt} and~\ref{r.llm}).
These results sit within a growing body of work bringing machine learning to number theory: classical algorithms have been applied to elliptic curves in \cite{ABH,HLOst,HLOac,HLOP,big-group-3} and deep learning in \cite{BCC+,hloq, babei2025machine}, with further applications to Galois groups \cite{HLOnf,LL}, class groups \cite{AHLOS}, Fricke signs \cite{big-group-1}, and vanishing orders of rational $L$-functions \cite{big-group-2}.

\Cref{s.background} reviews the necessary background.
\Cref{s.experiments} reports the decision tree experiments from which the formulae of Theorem~\ref{thm.main} were found.
\Cref{s.proofs} proves them: the prime $2$ governs $w_1$ and $w_3$, a single congruence at $p=3,5,7$ governs $w_2$, $w_4$ and $w_6$, and the geometric content of the formulae is recorded in \Cref{s.geom}.
\Cref{s.ectables} explores applications to computing tables of elliptic curves.
Finally, \Cref{a.transformers} revisits the problem using transformer architectures and provides a complementary machine-learning perspective.

\subsection*{Acknowledgments}
We thank John Cremona, Alyson Deines, Ken Ribet, and Vijay Srinivasan for comments and questions on an earlier version of the manuscript. 
This work grew out of discussions at the \emph{Machine Learning and Mathematics} conference held at the Korea Institute for Advanced Study (KIAS) in 2025, and was concluded during the \emph{AI and Mathematics} program at KIAS in 2026. The authors are grateful to the organizers for their hospitality and support.

\subsection*{Declaration of AI usage}
The authors made use of ChatGPT and Claude in order to encode and decode decision tree models. ChatGPT-5.2 Pro was used to help identify the formula for $w_2$ based on the decision-tree experiments, while Claude Code (with Claude Opus 4.8) assisted in identifying the formula for $w_3$ (see Remarks~\ref{rem:w2chatgpt} and~\ref{r.llm} for details). In addition, Claude Code (with Claude Opus 4.8) and Codex (with ChatGPT--5.6 Pro Sol) were used as coding assistants during the development of the Python/SageMath code.

\section{Weierstrass equations and Frobenius traces}\label{s.background}
In this preliminary section, we explain the relevant background theory using illustrative examples.

\begin{example}
The smallest possible value of $|\Delta|$ for an elliptic curve over $\mathbb{Q}$ is $11$.
In particular, the following reduced minimal Weierstrass equation has $\Delta = -11$:
\begin{equation}\label{eq.11a1}
y^2+y=x^3-x^2-7820x-263580.
\end{equation}
The curve in equation~\eqref{eq.11a1} is called \texttt{11\!.\!a1} in \cite{lmfdb}, and will appear again in later examples.
\end{example}
If $p$ is a prime integer and $E$ is an elliptic curve defined over $\mathbb{Q}$, then the reduction of $E$ at $p$ is the curve over $\mathbb{F}_p$ determined by reducing a minimal Weierstrass equation for $E$ modulo $p$.
We say that $p$ is a bad prime for $E$ if the reduction at $p$ is singular, in which case, there are two possible reduction types, namely, cuspidal (additive) and nodal (multiplicative).
In the nodal case, we say that the multiplicative reduction is split (resp. non-split) if the two tangents are (resp. are not) defined over $\mathbb{F}_p$.
The conductor $N(E)$ of an elliptic curve $E$ defined over $\mathbb{Q}$ is an integer divisible precisely by the primes of bad reduction, with the exponent at $p$ determined by the reduction type \cite[VIII.11,~C.16]{silverman}.
Szpiro's conjecture implies a bound on $|\Delta|$ in terms of $N(E)$ \cite[VIII.11]{silverman}.
\begin{example}
In some cases, the conductor may be equal to the absolute value of the discriminant. For example, the elliptic curve \texttt{11\!.\!a1} in equation~\eqref{eq.11a1} has conductor $11$.
On the other hand, these two quantities can be quite different.
For example, the following elliptic curve has conductor $11$ and discriminant $-11^5=-161051$:
\begin{equation}\label{eq.11a2}
y^2+y=x^3-x^2-10x-20.
\end{equation}
The curve in equation~\eqref{eq.11a2} is called \texttt{11\!.\!a2} on \cite{lmfdb}.
\end{example}
If $E$ is an elliptic curve defined over $\mathbb{Q}$, and $p$ is a good prime for $E$, the Frobenius trace of $E$ at $p$ is given by
\begin{equation}\label{eq.Frobenius}
a_p(E)=p+1-\#E(\mathbb{F}_p),
\end{equation}
where $\# E(\mathbb{F}_p)$ denotes the number of projective points for $E$ over $\mathbb{F}_p$.
When $p$ is a bad prime, we have $a_p(E)\in\{-1,0,1\}$ depending on the reduction type of $E$ at $p$.
For good primes we have Hasse's theorem that $|a_p(E)|\leq2\sqrt{p}$ \cite[V.1]{silverman}.
Given an elliptic curve, it is possible to compute the sequence of Frobenius traces by counting points on its reduction modulo primes.
Two elliptic curves over $\mathbb{Q}$ are isogenous if the corresponding sequences of Frobenius traces are equal at all good primes.
\begin{example}
The curves \texttt{11\!.\!a1} and \texttt{11\!.\!a2} in equations~\eqref{eq.11a1} and~\eqref{eq.11a2} are isogenous. The sequence of Frobenius traces begins: $a_2(E)=-2,~a_3(E)=-1,~a_5(E)=1,~a_7(E)=-2,~a_{11}(E)=1,~a_{13}(E)=4,~\dots$
\end{example}
Roughly speaking, we are interested in the inverse problem of constructing a reduced minimal Weierstrass equation from a finite sequence of Frobenius traces.
This problem is not well-posed for a number of reasons, for example, isogenous curves have the same Frobenius traces, and even non-isogenous curves can still agree on a finite list of Frobenius traces.
Nevertheless, we will show that $w_1,w_2,w_3$ depend only on 
%the isogeny class, and, in fact, only on 
$[a_2(E)]_2$, $[a_3(E)]_3$, $[N(E)]_2$, and, in the case that $[a_2(E)]_2=1$, the value $(a_2(E)+1)/2$, which is determined by $\mathrm{sgn}(a_2(E))$. 
The isogeneous curves \texttt{11\!.\!a1} and \texttt{11\!.\!a2} given by equations~\eqref{eq.11a1} and~\eqref{eq.11a2} demonstrate that $w_4$ and $w_6$ are not determined by the isogeny class alone.

The following two straightforward Lemmas will aid the subsequent exposition.
In order to state these Lemmas, we will need the notion of $p$-minimiality. 
We say that a Weierstrass equation is minimal at $p$ if all its coefficients are $p$-adic integers and, among all such equations, it achieves the minimal value of $v_p(\Delta)$.

\begin{lem}\label{lem.a2count}
Let $E/\Q$ be an elliptic curve, let $p$ be a prime, and let $\tilde{E}/\F_p$ be the
reduction modulo $p$ of any Weierstrass model of $E$ that is minimal at $p$. Then
\begin{equation}\label{eq.apbad}
a_p(E) \;=\; p + 1 - \#\tilde{E}(\F_p), \end{equation}
where $\#\tilde{E}(\F_p)$ counts all $\F_p$-points of the projective Weierstrass cubic,
including its singular point when $\tilde{E}$ is singular.
In particular, equation~\eqref{eq.apbad} is valid for the global minimal model of $E/\mathbb{Q}$.
\end{lem}
\begin{proof}
This follows from \cite[Proposition 2.5]{silverman} and \cite[Appendix C.16]{silverman}. Here we give a sketch of a proof.
If $\tilde{E}$ is smooth this is the definition of the trace of Frobenius. Otherwise
$\tilde{E}$ has a unique singular point, which, being unique, is fixed by Frobenius and
hence $\F_p$-rational. Since the model is minimal at $p$, the smooth locus
$\tilde{E}^{\mathrm{ns}}$ is isomorphic to $\mathbb{G}_m$, the nonsplit one-dimensional
torus, or $\mathbb{G}_a$, depending on whether $E$ has split multiplicative, nonsplit
multiplicative, or additive reduction at $p$; these have $p-1$, $p+1$ and $p$ points
over $\F_p$ respectively. Adding the singular point gives
$\#\tilde{E}(\F_p)\in\{p,\,p+2,\,p+1\}$, i.e.\ $a_p(E)=p+1-\#\tilde{E}(\F_p)\in\{+1,\,-1,\,0\}$ in the three cases.
\end{proof}

\begin{lem}\label{lem.Nparity}
Let $E/\Q$ be an elliptic curve, let $p$ be a prime, and let $\tilde{E}/\F_p$ be the
reduction modulo $p$ of any Weierstrass model of $E$ that is minimal at $p$. Then
$p \mid N(E)$ if and only if $\tilde{E}$ is singular. In particular, $[N(E)]_2 = 1$ if
and only if the reduction modulo $2$ of a model of $E$ that is minimal at $2$ is smooth.
\end{lem}
\begin{proof}
See \cite[Appendix C.16]{silverman}.
\end{proof}

\begin{lem}\label{lem.a4parity}
Let $E/\Q$ be an elliptic curve. Then
\[[N(E)]_2=\left[a_4(E)\left(a_4(E)-1\right)/2\right]_2;\]
equivalently, $N(E)$ is odd if and only if $a_4(E)\equiv2,3\pmod4$.
\end{lem}
\begin{proof}
Multiplicativity of the $a_n(E)$ and the Euler factor at $2$ give $a_4(E)=a_2(E)^2-2$ if $2\nmid N(E)$, and $a_4(E)=a_2(E)^2$ if $2\mid N(E)$.
Since a square is congruent to $0$ or $1$ modulo $4$, the first case puts $a_4(E)$ in $\{2,3\}$ and the second in $\{0,1\}$ modulo $4$.
The binomial coefficient $a_4(E)(a_4(E)-1)/2$ is odd precisely for the former pair.
\end{proof}

\section{Decision tree experiments}\label{s.experiments}
The formulae of Theorem~\ref{thm.main} were not guessed: they were read off decision tree models trained to predict the reduced minimal Weierstrass coefficients from Frobenius traces.
This section reports those experiments for all five coefficients; the proofs are deferred to \Cref{s.proofs}, and a reader interested only in the mathematics may skip ahead.
Throughout, we use the notation of equations~\eqref{eq.Weierstrass} and~\eqref{eq.Frobenius}.
All experiments involve a database $\mathcal{E}_1$ consisting of all isogeny classes of conductor $\leq10^5$ from the LMFDB \cite{lmfdb}; implementations are available at \cite{L26}.

\subsection{Predicting $w_1$}\label{s.w1}

\begin{example}\label{ex.a2w1}
In Figure~\ref{fig:dt_w1_a2}, we exhibit a decision tree which predicts $w_1$ from $(a_p(E))_{p<100}$ with $100\%$ accuracy. 
We observe that this decision tree uses only the first Frobenius trace $a_2(E)$.
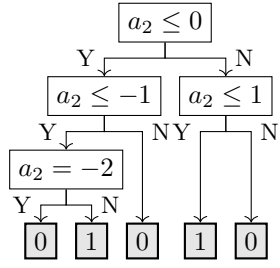
\begin{figure}
    \centering
    \footnotesize
    \begin{forest}
      label L/.style={
        edge label={node[midway,left,font=\scriptsize]{#1}}
      },
      label R/.style={
        edge label={node[midway,right,font=\scriptsize]{#1}}
      },
      for tree={
        forked edge,
        child anchor=north,
        for descendants={
          {edge=->}
        }
      },
      % [{$a_{2} \le 0.5$}, draw,
      %   [{$a_2 \le -0.5$}, draw, label L=Y,
      %       [{$a_2 \le -1.5$}, draw, label L=Y,
      %           [{$w_1 = 0$}, rectangle, thick, draw, label L=Y, tier=bottom, fill=gray!20]
      %           [{$w_1 = 1$}, rectangle, thick, draw, label R=N, tier=bottom, fill=gray!20]
      %       ]
      %       [{$w_1 = 0$}, rectangle, thick, draw, label R=N, tier=bottom, fill=gray!20]
      %   ]
      %   [{$a_2 \le 1.5$}, draw, label R=N,
      %       [{$w_1 = 1$}, rectangle, thick, draw, label L=Y, tier=bottom, fill=gray!20]
      %       [{$w_1 = 0$}, rectangle, thick, draw, label R=N, tier=bottom, fill=gray!20]
      %   ]
      % ]
        [{$a_{2} \le 0$}, draw,
            [{$a_2 \le -1$}, draw, label L=Y,
                [{$a_2 = -2$}, draw, label L=Y,
                    [{$0$}, rectangle, thick, draw, label L=Y, tier=bottom, fill=gray!20]
                    [{$1$}, rectangle, thick, draw, label R=N, tier=bottom, fill=gray!20]
                ]
                [{$0$}, rectangle, thick, draw, label R=N, tier=bottom, fill=gray!20]
            ]
            [{$a_2 \le 1$}, draw, label R=N,
                [{$1$}, rectangle, thick, draw, label L=Y, tier=bottom, fill=gray!20]
                [{$0$}, rectangle, thick, draw, label R=N, tier=bottom, fill=gray!20]
            ]
        ]
    \end{forest}
    \caption{A decision tree predicting $w_1$ using Frobenius traces $(a_p(E))_{p < 100}$ achieving 100\% accuracy. It only uses $a_2(E)$ for predictions.}
    \label{fig:dt_w1_a2}
\end{figure}
\end{example}

The tree of Figure~\ref{fig:dt_w1_a2} splits on $a_2(E)$ alone, and in fact only on its residue modulo $2$; this is \eqref{eq.thmw1}.

\begin{example}\label{ex.w1withouta2}
By way of contrast, in Table \ref{tab:dt_w1} we record the performance of decision trees to predict $w_1$ from $(a_{p}(E))_{3 \le p \le B}$, that is, without using $a_2(E)$ for $B \in \{10^2, 10^3, 10^4\}$. 
We observed accuracies of $51.03\%, 56.02\%$ and $62.39\%$ respectively, and that the resulting trees are highly complicated and hard to interpret.
\end{example}

\begin{table}[h]
    \centering
    \begin{tabular}{c|c|c|c}
    \toprule
    features & $p < 10^2$ & $p < 10^3$ & $p < 10^4$ \\
    \midrule
    all $a_p(E)$ & 100\% & 100\% & 100\% \\
    without $a_2(E)$ & 51.03\% & 56.02\% & 62.39\% \\
    \bottomrule
    \end{tabular}
    \caption{Performance of decision tree models on predicting $w_1$ in $\mathcal{E}_1$.}
    \label{tab:dt_w1}
\end{table}

If $w_1$ were to be uniformly distributed accross $\mathcal{E}_1$ then we would expect to see accuracy $1/2$, thereby explaining Table~\ref{tab:dt_w1}. 
By Theorem~\ref{thm.main}, uniform distribution of $w_1$ is equivalent to that of $a_2(E) \bmod 2$. 
Empirically, we observe the following counts in $\mathcal{E}_1$ as in Table~\ref{tab.w1count} \cite{L26}.
On the other hand the model does slightly better than random. 
There is a superficial reason for this, namely that, on finite datasets with small conductors, we may have artificial correlations or small numbers phenomena that will fade in the limit $N(E) \to \infty$. 
There is also a deeper reason that may suggest we should expect higher accuracy. 
Indeed, it is known that one can determine $a_2(E)$ from a sufficiently large number of $a_p(E)$ for $p > 2$ using the functional equation. 
An effective algorithm to do so is described in \cite[Section~5]{Boo}, which also shows that $\{a_p(E)\}_{p > 3}$ and $N(E) \bmod 2$ effectively determine $w_1,w_2,w_3$. 
Furthermore, it has been observed that machine learning models can learn to predict $a_2(E) \bmod 2$ non-trivially from larger $a_p(E)$ \cite{BCC+}. 
In fact, by Theorem~\ref{thm.main}, predicting $w_1$ without $a_2(E)$ in Table~\ref{tab:dt_w1} amounts to predicting $[a_2(E)]_2$  without $a_2(E)$, which is precisely the classification problem studied in \cite{BCC+}.

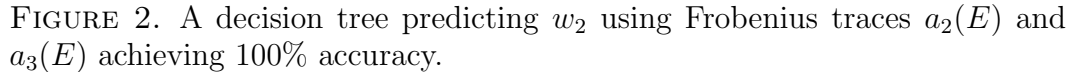
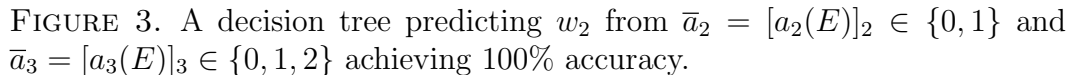
\begin{table}[h]
    \centering
\begin{tabular}{c|c}
\toprule
$\#\{E\in\mathcal{E}_1:w_1=0\}$ & $\#\{E\in\mathcal{E}_1:w_1=1\}$ \\
\midrule
219952 (50.3\%) & 217274 (49.7\%)\\
\bottomrule
\end{tabular}
\caption{Counting curves in $\mathcal{E}_1$ with each possible value for $w_1$.}
\label{tab.w1count}
\end{table}

\subsection{Predicting $w_2$}\label{s.w2}

\begin{example}\label{ex.a2a3w2}
In Figure~\ref{fig:dt_w2_a2_a3}, we exhibit a decision tree which predicts $w_2$ from $(a_2(E),a_3(E))$ with $100\%$ accuracy. 
\begin{figure}
    \centering
    \footnotesize
    \resizebox{\textwidth}{!}{%
    \begin{forest}
      label L/.style={
        edge label={node[midway,left,font=\scriptsize]{#1}}
      },
      label R/.style={
        edge label={node[midway,right,font=\scriptsize]{#1}}
      },
      for tree={
        forked edge,
        child anchor=north,
        for descendants={
          {edge=->}
        }
      },
      [{$a_{3} \le 0$}, draw,
          [{$a_{3} \le -1$}, draw, label L=Y,
              [{$a_{3} \le -2$}, draw, label L=Y,
                  [{$a_{3} \le -3$}, draw, label L=Y,
                      [{$a_{2} \le 0$}, draw, label L=Y,
                          [{$a_{2} \le -1$}, draw, label L=Y,
                              [{$a_{2} \le -2$}, draw, label L=Y,
                                  [{$0$}, rectangle, thick, draw, label L=Y, tier=bottom, fill=gray!20]
                                  [{$-1$}, rectangle, thick, draw, label R=N, tier=bottom, fill=gray!20]
                              ]
                              [{$0$}, rectangle, thick, draw, label R=N, tier=bottom, fill=gray!20]
                          ]
                          [{$a_{2} \le 1$}, draw, label R=N,
                              [{$-1$}, rectangle, thick, draw, label L=Y, tier=bottom, fill=gray!20]
                              [{$0$}, rectangle, thick, draw, label R=N, tier=bottom, fill=gray!20]
                          ]
                      ]
                      [{$a_{2} \le 0$}, draw, label R=N,
                          [{$a_{2} \le -1$}, draw, label L=Y,
                              [{$a_{2} \le -2$}, draw, label L=Y,
                                  [{$1$}, rectangle, thick, draw, label L=Y, tier=bottom, fill=gray!20]
                                  [{$0$}, rectangle, thick, draw, label R=N, tier=bottom, fill=gray!20]
                              ]
                              [{$1$}, rectangle, thick, draw, label R=N, tier=bottom, fill=gray!20]
                          ]
                          [{$a_{2} \le 1$}, draw, label R=N,
                              [{$0$}, rectangle, thick, draw, label L=Y, tier=bottom, fill=gray!20]
                              [{$1$}, rectangle, thick, draw, label R=N, tier=bottom, fill=gray!20]
                          ]
                      ]
                  ]
                  [{$a_{2} \le -1$}, draw, label R=N,
                      [{$a_{2} \le -2$}, draw, label L=Y,
                          [{$-1$}, rectangle, thick, draw, label L=Y, tier=bottom, fill=gray!20]
                          [{$1$}, rectangle, thick, draw, label R=N, tier=bottom, fill=gray!20]
                      ]
                      [{$a_{2} \le 0$}, draw, label R=N,
                          [{$-1$}, rectangle, thick, draw, label L=Y, tier=bottom, fill=gray!20]
                          [{$a_{2} \le 1$}, draw, label R=N,
                              [{$1$}, rectangle, thick, draw, label L=Y, tier=bottom, fill=gray!20]
                              [{$-1$}, rectangle, thick, draw, label R=N, tier=bottom, fill=gray!20]
                          ]
                      ]
                  ]
              ]
              [{$a_{2} \le -1$}, draw, label R=N,
                  [{$a_{2} \le -2$}, draw, label L=Y,
                      [{$0$}, rectangle, thick, draw, label L=Y, tier=bottom, fill=gray!20]
                      [{$-1$}, rectangle, thick, draw, label R=N, tier=bottom, fill=gray!20]
                  ]
                  [{$a_{2} \le 0$}, draw, label R=N,
                      [{$0$}, rectangle, thick, draw, label L=Y, tier=bottom, fill=gray!20]
                      [{$a_{2} \le 1$}, draw, label R=N,
                          [{$-1$}, rectangle, thick, draw, label L=Y, tier=bottom, fill=gray!20]
                          [{$0$}, rectangle, thick, draw, label R=N, tier=bottom, fill=gray!20]
                      ]
                  ]
              ]
          ]
          [{$a_{3} \le 1$}, draw, label R=N,
              [{$a_{2} \le 0$}, draw, label L=Y,
                  [{$a_{2} \le -1$}, draw, label L=Y,
                      [{$a_{2} \le -2$}, draw, label L=Y,
                          [{$1$}, rectangle, thick, draw, label L=Y, tier=bottom, fill=gray!20]
                          [{$0$}, rectangle, thick, draw, label R=N, tier=bottom, fill=gray!20]
                      ]
                      [{$1$}, rectangle, thick, draw, label R=N, tier=bottom, fill=gray!20]
                  ]
                  [{$a_{2} \le 1$}, draw, label R=N,
                      [{$0$}, rectangle, thick, draw, label L=Y, tier=bottom, fill=gray!20]
                      [{$1$}, rectangle, thick, draw, label R=N, tier=bottom, fill=gray!20]
                  ]
              ]
              [{$a_{3} \le 2$}, draw, label R=N,
                  [{$a_{2} \le -1$}, draw, label L=Y,
                      [{$a_{2} \le -2$}, draw, label L=Y,
                          [{$-1$}, rectangle, thick, draw, label L=Y, tier=bottom, fill=gray!20]
                          [{$1$}, rectangle, thick, draw, label R=N, tier=bottom, fill=gray!20]
                      ]
                      [{$a_{2} \le 0$}, draw, label R=N,
                          [{$-1$}, rectangle, thick, draw, label L=Y, tier=bottom, fill=gray!20]
                          [{$a_{2} \le 1$}, draw, label R=N,
                              [{$1$}, rectangle, thick, draw, label L=Y, tier=bottom, fill=gray!20]
                              [{$-1$}, rectangle, thick, draw, label R=N, tier=bottom, fill=gray!20]
                          ]
                      ]
                  ]
                  [{$a_{2} \le -1$}, draw, label R=N,
                      [{$a_{2} \le -2$}, draw, label L=Y,
                          [{$0$}, rectangle, thick, draw, label L=Y, tier=bottom, fill=gray!20]
                          [{$-1$}, rectangle, thick, draw, label R=N, tier=bottom, fill=gray!20]
                      ]
                      [{$a_{2} \le 0$}, draw, label R=N,
                          [{$0$}, rectangle, thick, draw, label L=Y, tier=bottom, fill=gray!20]
                          [{$a_{2} \le 1$}, draw, label R=N,
                              [{$-1$}, rectangle, thick, draw, label L=Y, tier=bottom, fill=gray!20]
                              [{$0$}, rectangle, thick, draw, label R=N, tier=bottom, fill=gray!20]
                          ]
                      ]
                  ]
              ]
          ]
      ]
    \end{forest}
    }
    \caption{A decision tree predicting $w_2$ using Frobenius traces $a_2(E)$ and $a_3(E)$ achieving 100\% accuracy. }
    \label{fig:dt_w2_a2_a3}
\end{figure}
Inspired by \eqref{eq.thmw1}, in Figure~\ref{fig:dt_a2mod2_a3mod3}, we exhibit a simpler decision tree which predicts $w_2$ from $([a_2(E)]_2,[a_3(E)]_3)$, again with $100\%$ accuracy. 
\begin{figure}
    \centering
    \footnotesize
    \begin{forest}
      label L/.style={
        edge label={node[midway,left,font=\scriptsize]{#1}}
      },
      label R/.style={
        edge label={node[midway,right,font=\scriptsize]{#1}}
      },
      for tree={
        forked edge,
        child anchor=north,
        for descendants={
          {edge=->}
        }
      },
      [{$\overline{a}_3 = 0$}, draw,
          [{$\overline{a}_2 = 0$}, draw, label L=Y,
              [{$0$}, rectangle, thick, draw, label L=Y, tier=bottom, fill=gray!20]
              [{$-1$}, rectangle, thick, draw, label R=N, tier=bottom, fill=gray!20]
          ]
          [{$\overline{a}_3 \le 1$}, draw, label R=N,
              [{$\overline{a}_2 = 0$}, draw, label L=Y,
                  [{$1$}, rectangle, thick, draw, label L=Y, tier=bottom, fill=gray!20]
                  [{$0$}, rectangle, thick, draw, label R=N, tier=bottom, fill=gray!20]
              ]
              [{$\overline{a}_2 = 0$}, draw, label R=N,
                  [{$-1$}, rectangle, thick, draw, label L=Y, tier=bottom, fill=gray!20]
                  [{$1$}, rectangle, thick, draw, label R=N, tier=bottom, fill=gray!20]
              ]
          ]
      ]
    \end{forest}
    \caption{A decision tree predicting $w_2$ from $\overline{a}_2 = [a_2(E)]_2 \in \{0, 1\}$ and $\overline{a}_3 = [a_3(E)]_3 \in \{0, 1, 2\}$ achieving 100\% accuracy.}
    \label{fig:dt_a2mod2_a3mod3}
\end{figure}
\end{example}

The tree of Figure~\ref{fig:dt_a2mod2_a3mod3} splits only on $[a_2(E)]_2$ and $[a_3(E)]_3$, which is \eqref{eq.thmw2}.

\begin{rem}\label{rem:w2chatgpt}
    The formula \eqref{eq.thmw2} was found with help of ChatGPT-5.2 Pro. 
    To do so, we generated a look-up table for the mapping $(a_2(E), a_3(E)) \mapsto w_2$ based on the decision tree in Figure \ref{fig:dt_w2_a2_a3}, and asked ChatGPT-5.2 Pro to find a simple formula describing the table.
\end{rem}

The following Example is inspired by Example~\ref{ex.w1withouta2}.
\begin{example}\label{ex.w2withouta2a3}
In Table~\ref{tab:dt_w2} we record the performance of decision trees to predict $w_2$ from vectors of the form $(a_{p}(E))_{3 \le p \le B}$, $(a_{p}(E))_{2 \le p \le B, p\neq3}$, and $(a_{p}(E))_{5 \le p \le B}$, that is, without using $a_2(E)$, $a_3(E)$, or both, for $B \in\{10^2, 10^3,10^4\}$. 
Without using $a_2(E)$ (resp. $a_3(E)$), we never achieve $>60\%$ (resp. $>40\%$) accuracy.
\end{example}

\begin{table}[h]
    \centering
    \begin{tabular}{c|c|c|c}
        \toprule
        features & $p < 10^2$ & $p < 10^3$ & $p < 10^4$ \\
        \midrule
        all $a_p(E)$ & 100\% & 100\% & 100\% \\
        without $a_2(E)$ & 50.92\% & 52.49\% & 54.27\% \\
        without $a_3(E)$ & 34.63\% & 36.20\% & 38.48\% \\
        without $a_2(E)$ and $a_3(E)$ & 33.66\% & 33.80\% & 33.73\% \\
        \bottomrule
    \end{tabular}
    \caption{Performance of decision tree models predicting $w_2$ in $\mathcal{E}_1$.}\label{tab:dt_w2}
\end{table}

If $w_2$ were to be uniformly distributed across $\mathcal{E}_1$ then we would expect to see accuracy $1/3$, thereby explaining Table~\ref{tab:dt_w2}. 
Empirically, we observe the following counts in $\mathcal{E}_1$ as in Table~\ref{tab.w2count} \cite{L26}.
As in the case of $w_1$, the model does slightly better than random, and that can be related to \cite{Boo} and \cite{BCC+}.

\begin{table}[h]
    \centering
\begin{tabular}{c|c|c}
\toprule
$\#\{E\in\mathcal{E}_1:w_2=-1\}$ & $\#\{E\in\mathcal{E}_1:w_2=0\}$ & $\#\{E\in\mathcal{E}_1:w_2=1\}$ \\
\midrule
154308 (35.3\%) & 154994 (35.4\%) & 127924 (29.3\%) \\
\bottomrule
\end{tabular}
\caption{Counting curves in $\mathcal{E}_1$ with each possible value for $w_2$.}
\label{tab.w2count}
\end{table}

\subsection{Predicting $w_3$}\label{s.w3}

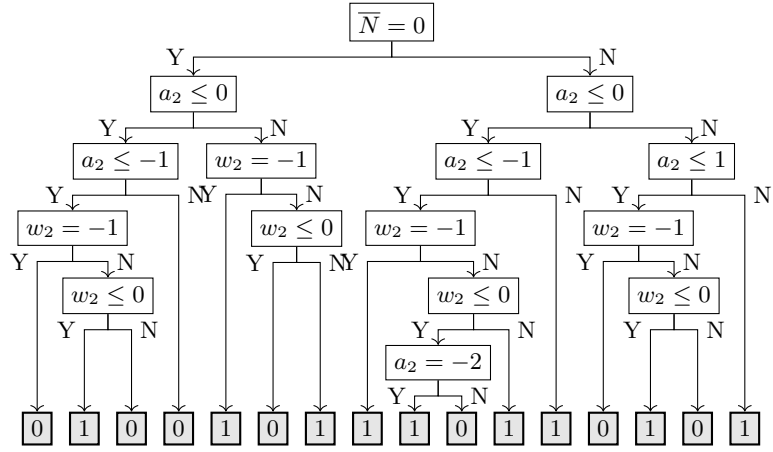
\begin{figure}
    \centering
    \scriptsize
    % \resizebox{\textwidth}{!}{%
    \begin{forest}
      label L/.style={
        edge label={node[midway,left,font=\scriptsize]{#1}}
      },
      label R/.style={
        edge label={node[midway,right,font=\scriptsize]{#1}}
      },
      leaf/.style={
        rectangle,
        thick,
        draw,
        fill=gray!20,
        tier=bottom
      },
      for tree={
        forked edge,
        child anchor=north,
        draw,
        for descendants={
          {edge=->}
        }
      },
      % [{$\overline{N} \le 0.5$}
      %   [{$a_2 \le 0.5$}, label L=Y
      %     [{$a_2 \le -0.5$}, label L=Y
      %       [{$w_2 \le -0.5$}, label L=Y
      %         [{$w_3=0$}, leaf, label L=Y]
      %         [{$w_2 \le 0.5$}, label R=N
      %           [{$w_3=1$}, leaf, label L=Y]
      %           [{$w_3=0$}, leaf, label R=N]
      %         ]
      %       ]
      %       [{$w_3=0$}, leaf, label R=N]
      %     ]
      %     [{$w_2 \le -0.5$}, label R=N
      %       [{$w_3=1$}, leaf, label L=Y]
      %       [{$w_2 \le 0.5$}, label R=N
      %         [{$w_3=0$}, leaf, label L=Y]
      %         [{$w_3=1$}, leaf, label R=N]
      %       ]
      %     ]
      %   ]
      %   [{$a_2 \le 0.5$}, label R=N
      %     [{$a_2 \le -0.5$}, label L=Y
      %       [{$w_2 \le -0.5$}, label L=Y
      %         [{$w_3=1$}, leaf, label L=Y]
      %         [{$w_2 \le 0.5$}, label R=N
      %           [{$a_2 \le -1.5$}, label L=Y
      %             [{$w_3=1$}, leaf, label L=Y]
      %             [{$w_3=0$}, leaf, label R=N]
      %           ]
      %           [{$w_3=1$}, leaf, label R=N]
      %         ]
      %       ]
      %       [{$w_3=1$}, leaf, label R=N]
      %     ]
      %     [{$a_2 \le 1.5$}, label R=N
      %       [{$w_2 \le -0.5$}, label L=Y
      %         [{$w_3=0$}, leaf, label L=Y]
      %         [{$w_2 \le 0.5$}, label R=N
      %           [{$w_3=1$}, leaf, label L=Y]
      %           [{$w_3=0$}, leaf, label R=N]
      %         ]
      %       ]
      %       [{$w_3=1$}, leaf, label R=N]
      %     ]
      %   ]
      % ]
      [{$\overline{N} = 0$}
        [{$a_2 \le 0$}, label L=Y
          [{$a_2 \le -1$}, label L=Y
            [{$w_2 = -1$}, label L=Y
              [{$0$}, leaf, label L=Y]
              [{$w_2 \le 0$}, label R=N
                [{$1$}, leaf, label L=Y]
                [{$0$}, leaf, label R=N]
              ]
            ]
            [{$0$}, leaf, label R=N]
          ]
          [{$w_2 = -1$}, label R=N
            [{$1$}, leaf, label L=Y]
            [{$w_2 \le 0$}, label R=N
              [{$0$}, leaf, label L=Y]
              [{$1$}, leaf, label R=N]
            ]
          ]
        ]
        [{$a_2 \le 0$}, label R=N
          [{$a_2 \le -1$}, label L=Y
            [{$w_2 = -1$}, label L=Y
              [{$1$}, leaf, label L=Y]
              [{$w_2 \le 0$}, label R=N
                [{$a_2 = -2$}, label L=Y
                  [{$1$}, leaf, label L=Y]
                  [{$0$}, leaf, label R=N]
                ]
                [{$1$}, leaf, label R=N]
              ]
            ]
            [{$1$}, leaf, label R=N]
          ]
          [{$a_2 \le 1$}, label R=N
            [{$w_2 = -1$}, label L=Y
              [{$0$}, leaf, label L=Y]
              [{$w_2 \le 0$}, label R=N
                [{$1$}, leaf, label L=Y]
                [{$0$}, leaf, label R=N]
              ]
            ]
            [{$1$}, leaf, label R=N]
          ]
        ]
      ]
    \end{forest}%
    % }
    \caption{A decision tree predicting $w_3$ in $\mathcal{E}_1$ from $\overline{N} = [N(E)]_2$, $a_2(E)$, and $w_2$.}
    \label{fig:dt_nmod2_a2_w2}
\end{figure}

\begin{example}\label{ex.na2w2w3}
In Figure \ref{fig:dt_nmod2_a2_w2} we exhibit a decision tree which predicts $w_3$ from  $\overline{N} = [N(E)]_2 \in \{0,1\}$, $a_2(E)$, and $w_2$ with $100\%$ accuracy. 
\end{example}

\begin{rem}\label{r.llm}
Formula \eqref{eq.w3}, prior to its reformulation as \eqref{eq.thmw3}, was discovered empirically in a conversation with Anthropic's Claude language model.
Motivated by \eqref{eq.thmw1} and \eqref{eq.thmw2}, we tabulated the triples $(w_1,w_2,w_3)$ together with the local reduction data of $E$ at $2$ and $3$ for all $657{,}396$ elliptic curves of conductor less than $100{,}000$ in the LMFDB. The resulting data showed that the reduction types at $2$ and $3$ (refined by the split/nonsplit and ordinary/supersingular distinctions) determine $(w_1,w_2,w_3)$, except that one must also specify $a_2(E)$ when $E$ has good ordinary reduction at $2$, and $a_3(E)$ when $E$ has good ordinary reduction at $3$. This reflects the fact that, at a good ordinary prime $p$, the residue class $a_p(E)\bmod p$ determines which Frobenius eigenvalue is the $p$-adic unit.
The case distinction in \eqref{eq.w3} is therefore explained by parity: $a_2(E)$ is odd precisely when $E$ has good ordinary or multiplicative reduction at $2$, and even precisely when it has good supersingular or additive reduction. The explicit formula was obtained by fitting the tabulated data, verified on the full dataset, and only subsequently proved.
Note that equation~\eqref{eq.w3} is also discoverable from the decision tree in Figure \ref{fig:dt_nmod2_a2_w2}. 
When $a_2(E)$ is even, Figure \ref{fig:dt_a2mod2_a3mod3} shows that $w_3$ is equal to $[N(E)]_2$. 
When $a_2(E) \in \{\pm 1\}$, $w_3$ depends on $a_2(E)$ and $[w_2]_2$, and the second part of equation \eqref{eq.w3} can be found. 
\end{rem}

\begin{example}
In Table~\ref{tab:dtw3}, we record the best accuracy achieved by a decision tree trained to predict $w_3$ from $(a_p(E))_{p<B}$, for $B\in\{10^2,10^3,10^4\}$.
In Table~\ref{tab.w3count}, we record the counts for each value of $w_3$ in $\mathcal{E}_1$.
\end{example}

\begin{table}[h]
    \centering
    \begin{tabular}{c|c|c|c}
        \toprule
        features & $p < 10^2$ & $p < 10^3$ & $p < 10^4$ \\
        \midrule
        all $a_p(E)$ & 73.71\% & 75.36\% & 76.68\% \\
        \bottomrule
    \end{tabular}
        \caption{Performance of decision tree models predicting $w_3$ in $\mathcal{E}_1$.}\label{tab:dtw3}
\end{table}

\begin{table}[h]
    \centering
\begin{tabular}{c|c}
\toprule
$\#\{E\in\mathcal{E}_1:w_3=0\}$ & $\#\{E\in\mathcal{E}_1:w_3=1\}$ \\
\midrule
283609 (64.9\%) & 153617 (35.1\%)\\
\bottomrule
\end{tabular}
\caption{Counting curves in $\mathcal{E}_1$ with each possible value for $w_3$.}
\label{tab.w3count}
\end{table}

\subsection{Predicting $w_4$ and $w_6$}\label{s.w4w6}
Unlike $w_1,w_2,w_3$, the coefficients $w_4$ and $w_6$ are not isogeny invariants, and no decision tree recovers them exactly.
We therefore train trees on their residues modulo small primes; the accuracies below single out $\ell=5$ for $w_4$ and $\ell=7$ for $w_6$, yielding \eqref{eq.thmw4} and \eqref{eq.thmw6}, respectively.

% \todoi{New experiments on $w_4$ and $w_6$.}

\begin{table}[h]
    \centering
    \begin{tabular}{c|c|c|c|c}
        \toprule
        $\ell$ & $2$ & $3$ & $5$ & $7$ \\
        \midrule
        all $a_p(E)$, $w_1(E)$, $w_2(E)$, $w_3(E)$ & 59.14\% & 49.01\% & 99.96\% & 24.54\% \\
        \bottomrule
    \end{tabular}
    \caption{Performance of decision tree models predicting $w_4(E) \bmod{\ell}$ for $\ell = 2, 3, 5, 7$ using $a_p(E)$ for $p < 10^5$ and $w_1(E), w_2(E), w_3(E)$ in $\mathcal{E}_1$.}\label{tab:dt_w4_ell}
\end{table}

\begin{table}[h]
    \centering
    \begin{tabular}{c|c}
        \toprule
        features & accuracy \\
        \midrule
        all $a_p(E)$, $w_1(E)$, $w_2(E)$, $w_3(E)$ & 99.96\% \\
        $a_5(E)$, $w_1(E)$, $w_2(E)$, $w_3(E)$ & 100\% \\
        all $a_p(E)$ but $a_5(E)$, $w_1(E)$, $w_2(E)$, $w_3(E)$ & 25.22\% \\
        all $a_p(E)$ & 82.80\% \\
        \bottomrule
    \end{tabular}
    \caption{Performance of decision tree models predicting $w_4(E) \bmod{5}$ in $\mathcal{E}_1$.}\label{tab:dt_w4mod5}
\end{table}

\begin{table}[h]
    \centering
    \begin{tabular}{c|c|c|c|c}
        \toprule
        $\ell$ & $2$ & $3$ & $5$ & $7$ \\
        \midrule
        all $a_p(E)$, $w_1(E)$, $w_2(E)$, $w_3(E)$ & 70.27\% & 45.32\% & 38.87\% & 40.52\% \\
        all $a_p(E)$, $w_1(E)$, $w_2(E)$, $w_3(E)$, $w_4(E) \bmod{\ell}$ & 88.43\% & 60.84\% & 39.08\% & 99.21\% \\
        \bottomrule
    \end{tabular}
    \caption{Performance of decision tree models predicting $w_6(E) \bmod{\ell}$ for $\ell = 2, 3, 5, 7$ using $a_p(E)$ for $p < 10^5$ and $w_1(E), w_2(E), w_3(E)$ in $\mathcal{E}_1$.}\label{tab:dt_w6_ell}
\end{table}

\begin{table}[h]
    \centering
    \begin{tabular}{c|c}
        \toprule
        features & accuracy \\
        \midrule
        all $a_p(E)$, $w_1(E)$, $w_2(E)$, $w_3(E)$, $w_4(E) \bmod{7}$ & 99.21\% \\
        $a_7(E)$, $w_1(E)$, $w_2(E)$, $w_3(E)$, $w_4(E) \bmod{7}$ & 100\% \\
        all $a_p(E)$ but $a_7(E)$, $w_1(E)$, $w_2(E)$, $w_3(E)$, $w_4(E) \bmod{7}$ & 19.17\% \\
        all $a_p(E)$ & 33.96\% \\
        \bottomrule
    \end{tabular}
    \caption{Performance of decision tree models predicting $w_6(E) \bmod{7}$ in $\mathcal{E}_1$.}\label{tab:dt_w6mod7}
\end{table}

Table \ref{tab:dt_w4_ell} shows experimental results on predicting $w_4(E) \bmod{\ell}$ with $(a_p(E))_{2 \le p \le 10^5}$, $w_1(E)$, $w_2(E)$, $w_3(E)$ as inputs.
It is clear that the decision tree can predict well especially when $\ell = 5$.
We performed extra experiments with different set of features for $\ell = 5$, where the results can be found in Table \ref{tab:dt_w4mod5}.
With four features $a_5(E), w_1(E), w_2(E), w_3(E)$, the decision tree can achieve 100\% of accuracy.

Table \ref{tab:dt_w6_ell} shows experimental results on predicting $w_6(E) \bmod{\ell}$.
The experiment with $\ell = 7$ and features $(a_p(E))_{2 \le p \le 10^5}$, $w_1(E)$, $w_2(E)$, $w_3(E)$, and $w_4(E) \bmod{7}$ stands out.
Table \ref{tab:dt_w6mod7} shows results from additional experiments with slightly different set of features, where the decision tree can achieve 100\% accuracy with five features $a_7(E)$, $w_1(E)$, $w_2(E)$, $w_3(E)$, and $w_4(E) \bmod{7}$.

\section{From Frobenius traces to Weierstrass coefficients}\label{s.proofs}
We now prove Theorem~\ref{thm.main}.
The five formulae fall into two groups, and the division is not the one between $\{w_1,w_2,w_3\}$ and $\{w_4,w_6\}$.
At an odd prime $p$ one may complete the square, and a single congruence for $a_p(E)$ (\Cref{lem.charsum}) yields $w_2$, $w_4$ and $w_6$ at $p=3$, $5$ and $7$ respectively.
At $p=2$ this is unavailable, and $w_1$ and $w_3$ --- precisely the coefficients that can be removed once $2$ is invertible --- are read off instead from two parities of the special fibre: that of its point count and that of the discriminant.
We treat the two cases in \Cref{s.p2} and \Cref{s.oddp}, reunite them in Remark~\ref{rem.vijay}, and record the geometric content of the formulae in \Cref{s.geom}.
Throughout, $b_2$, $b_4$, $b_6$, $b_8$ and $\Delta$ denote the standard quantities attached to \eqref{eq.Weierstrass} \cite[III.1]{silverman}:
\begin{gather*}
b_2=w_1^2+4w_2,\qquad b_4=2w_4+w_1w_3,\qquad b_6=w_3^2+4w_6,\\
b_8=w_1^2w_6+4w_2w_6-w_1w_3w_4+w_2w_3^2-w_4^2,\qquad
\Delta=-b_2^2b_8-8b_4^3-27b_6^2+9b_2b_4b_6.
\end{gather*}

\subsection{The prime $2$: $w_1$ and $w_3$}\label{s.p2}
Throughout this subsection $E/\Q$ is written in its reduced minimal Weierstrass model \eqref{eq.Weierstrass}, which is in particular minimal at $2$, and $\tilde{E}/\F_2$ denotes its reduction modulo $2$, with coefficients $\tilde{w}_i\in\F_2$.
By \eqref{eq.reducedW} we have $w_1=\tilde{w}_1$ and $w_3=\tilde{w}_3$ under the identification of $\{0,1\}\subset\Z$ with $\F_2$, and $w_i\equiv\tilde{w}_i \pmod 2$ for $i=2,4,6$.
The curve $\tilde{E}$ may be singular; by \Cref{lem.a2count}, $a_2(E)=3-\#\tilde{E}(\F_2)$ with the singular point included in the count.
Moreover $2\mid N(E)$ if and only if $\tilde{E}$ is singular (\Cref{lem.Nparity}), if and only if $2\mid\Delta$; briefly, $[N(E)]_2=[\Delta]_2$, and by \Cref{lem.a4parity} this common value is the quantity $\varepsilon$ of Theorem~\ref{thm.main}.

\begin{thm}\label{lem.w1a2}
If $E$ is an elliptic curve over $\mathbb{Q}$, then the reduced minimal Weierstrass coefficient $w_1$ satisfies $w_1= [a_2(E)]_2$.
\end{thm}

\begin{proof}
Since $y^2=y$ on $\F_2$, the reduced equation is \emph{linear} in $y$: the fibre of $\tilde{E}$ above $x\in\F_2$ is cut out by
\begin{equation}\label{eq.linearfibre}
\left(1+\tilde{w}_1x+\tilde{w}_3\right)y=x^3+\tilde{w}_2x^2+\tilde{w}_4x+\tilde{w}_6.
\end{equation}
An equation $cy=r$ over $\F_2$ has one solution when $c=1$, and zero or two when $c=0$; in all cases the number of solutions is congruent to $c$ modulo $2$.
Summing $1+\tilde{w}_1x+\tilde{w}_3$ over $x\in\F_2$ and adding the point at infinity gives
\[\#\tilde{E}(\F_2)\equiv1+\tilde{w}_1\pmod 2,\]
so $a_2(E)=3-\#\tilde{E}(\F_2)\equiv\tilde{w}_1 \pmod 2$.
\end{proof}

\begin{rem}
The curve \texttt{11\!.\!a2} has good reduction at $2$, but also admits the non-minimal model
$y^2 + 8y = x^3 - 4x^2 - 160x - 1280$, which reduces modulo $2$ to the cuspidal cubic $y^2 = x^3$.
Since $a_2(\texttt{11\!.\!a2}) = -2$, the point count of this special fibre does not compute $a_2$.
The point is that the reduction type of $E$ at $2$ can be read off from the special fibre only for a model that is minimal at $2$.
\end{rem}

\begin{thm}\label{thm.w3}
If $E$ is an elliptic curve over $\mathbb{Q}$, then the reduced minimal Weierstrass coefficient $w_3$ satisfies
\begin{equation}\label{eq:w3short}
    w_3 = (1 - [a_2(E)]_2) [N(E)]_2 + [a_2(E)]_2 \left[ N(E) + 1 + w_2 + \frac{a_2(E) + 1}{2} \right]_2.
\end{equation}
Equivalently, we have:
\begin{equation}\label{eq.w3}
w_3=\begin{cases}
[N(E)]_2 & \text{ if }\  a_2(E)\in\{-2,0,2\};\\
\left[N(E)+1 +w_2+\frac{a_2(E)+1}{2}\right]_2 &  \text{ if }\ a_2(E)\in\{-1,1\}.\end{cases}
\end{equation}
\end{thm}

\begin{proof}
Reducing $\Delta$ modulo $2$ using $b_2\equiv w_1$, $b_4\equiv w_1w_3$ and $b_6\equiv w_3$ gives
\begin{equation}\label{eq.deltaparity}
\Delta\equiv w_1b_8+w_3+w_1w_3 \ \pmod 2,
\end{equation}
and by Theorem~\ref{lem.w1a2} the two cases of \eqref{eq.w3} are $w_1=0$ and $w_1=1$.

Suppose first that $w_1=0$.
Then \eqref{eq.deltaparity} reads $w_3\equiv\Delta\equiv N(E) \pmod 2$, which is the first case of \eqref{eq.w3}.

Suppose now that $w_1=1$, so that the coefficient $1+x+\tilde{w}_3$ of $y$ in \eqref{eq.linearfibre} vanishes at exactly one $x_0\in\F_2$, namely $x_0=1+\tilde{w}_3$.
Above the other value of $x$ the fibre has exactly one point, while above $x_0$ it has two or zero points according as
\[r\coloneqq x_0^3+\tilde{w}_2x_0^2+\tilde{w}_4x_0+\tilde{w}_6=\left(1+\tilde{w}_3\right)\left(1+\tilde{w}_2+\tilde{w}_4\right)+\tilde{w}_6\in\F_2\]
is $0$ or $1$; here we used $x_0^m=x_0$.
With the point at infinity, $\#\tilde{E}(\F_2)=4-2r$, so $a_2(E)=2r-1\in\{-1,1\}$ and
\begin{equation}\label{eq.a2odd}
\frac{a_2(E)+1}{2}\equiv(1+w_3)(1+w_2+w_4)+w_6 \ \pmod 2.
\end{equation}
On the other hand, when $w_1=1$ we have $b_8\equiv w_2w_3+(1+w_3)w_4+w_6 \pmod 2$, so \eqref{eq.deltaparity} becomes
\begin{equation}\label{eq.Ncondodd}
N(E)\equiv\Delta\equiv w_2w_3+(1+w_3)w_4+w_6 \ \pmod 2.
\end{equation}
Adding \eqref{eq.a2odd} and \eqref{eq.Ncondodd} kills $w_4$ and $w_6$:
\[\frac{a_2(E)+1}{2}+N(E)\equiv(1+w_3)(1+w_2)+w_2w_3\equiv1+w_2+w_3 \ \pmod 2,\]
which rearranges to the second case of \eqref{eq.w3}.
\end{proof}

\subsection{The primes $3$, $5$ and $7$: $w_2$, $w_4$ and $w_6$}\label{s.oddp}
For odd $p$ the substitution $Y=2y+w_1x+w_3$ turns \eqref{eq.Weierstrass} into $Y^2=f(x)$, where
\[f(x)=4x^3+b_2x^2+2b_4x+b_6.\]
The following congruence is the engine behind all three of the remaining formulae.
Here $[x^k]F$ denotes the coefficient of $x^k$ in a polynomial $F$.

\begin{lem}\label{lem.charsum}
Let $p$ be an odd prime and let $E/\Q$ be given by a Weierstrass model that is minimal at $p$. Then
\[a_p(E)\equiv\big[x^{p-1}\big]\,f(x)^{(p-1)/2} \ \pmod p.\]
\end{lem}

\begin{proof}
Let $\chi$ be the quadratic character of $\F_p$, extended by $\chi(0)=0$.
The substitution above is invertible over $\F_p$, so the fibre of $\tilde{E}$ above $x\in\F_p$ has $1+\chi(f(x))$ points, whatever the reduction type: when $f(x)=0$ this counts the point with $Y=0$, which is the singular point if $\tilde{E}$ is singular there.
Summing over $x$ and adding the point at infinity gives $\#\tilde{E}(\F_p)=p+1+\sum_{x\in\F_p}\chi(f(x))$, so $a_p(E)=-\sum_{x\in\F_p}\chi(f(x))$ by \Cref{lem.a2count}.
By Euler's criterion, $\chi(t)\equiv t^{(p-1)/2} \pmod p$ for all $t\in\F_p$, whence
\[a_p(E)\equiv-\sum_{x\in\F_p}f(x)^{(p-1)/2} \ \pmod p.\]
Finally, $\sum_{x\in\F_p}x^k$ is $-1$ if $k>0$ is divisible by $p-1$ and $0$ otherwise, and $\deg f^{(p-1)/2}=\tfrac{3}{2}(p-1)<2(p-1)$, so the only exponent that contributes is $k=p-1$, with sign $-1$.
\end{proof}

Taking $p=3$, $5$ and $7$ in \Cref{lem.charsum} amounts to extracting the coefficients
\[
[x^{2}]f=b_2,\qquad
[x^{4}]f^{2}=b_2^2+16b_4,\qquad
[x^{6}]f^{3}=b_2^3+48b_2b_4+48b_6,
\]
which, reduced modulo $3$, $5$ and $7$ respectively, give
\begin{equation}\label{eq.bladder}
a_3(E)\equiv b_2 \pmod{3},\quad a_5(E)\equiv b_2^2+b_4 \pmod{5},\quad a_7(E)\equiv b_2^3-b_2b_4-b_6 \pmod{7}.
\end{equation}
Each of the three formulae below follows by solving \eqref{eq.bladder} for the highest $b$-invariant occurring.

\begin{thm}\label{lem.formulaw2}
If $E$ is an elliptic curve over $\mathbb{Q}$, then the reduced minimal Weierstrass coefficient $w_2$ satisfies
\begin{equation}\label{eq.w2a3w1}
w_2
=[a_3(E)+1-w_1]_3-1.
\end{equation}
\end{thm}

\begin{proof}
By \eqref{eq.bladder} we have $a_3(E)\equiv b_2=w_1^2+4w_2\equiv w_1+w_2\pmod 3$, since $w_1\in\{0,1\}$ gives $w_1^2=w_1$.
As $w_2\in\{-1,0,1\}$ is a complete set of residues modulo $3$, it is determined by $w_2\equiv a_3(E)-w_1 \pmod 3$, which is \eqref{eq.w2a3w1}.
\end{proof}

\begin{thm}\label{thm.w4}
If $E$ is an elliptic curve over $\mathbb{Q}$, then the reduced minimal Weierstrass coefficient $w_4$ satisfies
\begin{equation}\label{eq.w4}
w_4\equiv3\left(a_5(E)-b_2^2-w_1w_3\right)\pmod 5 .
\end{equation}
\end{thm}

\begin{proof}
By \eqref{eq.bladder}, $b_4\equiv a_5(E)-b_2^2\pmod 5$. Substituting $b_4=2w_4+w_1w_3$ and multiplying by $2^{-1}=3$ in $\F_5$ gives \eqref{eq.w4}.
\end{proof}

\begin{thm}\label{thm.w6}
If $E$ is an elliptic curve over $\mathbb{Q}$, then the reduced minimal Weierstrass coefficient $w_6$ satisfies
\begin{equation}\label{eq.w6}
w_6\equiv5\left(a_7(E)-b_2^3+b_2w_1w_3+w_3^2\right)+3b_2w_4\pmod 7 .
\end{equation}
\end{thm}

\begin{proof}
By \eqref{eq.bladder}, $b_6\equiv b_2^3-b_2b_4-a_7(E)\pmod 7$. Substituting $b_6=w_3^2+4w_6$ and $b_4=2w_4+w_1w_3$ and multiplying by $4^{-1}=2$ in $\F_7$ gives \eqref{eq.w6}.
\end{proof}

\begin{rem}\label{rem.ladder}
Equation~\eqref{eq.bladder} explains why the list stops at $7$.
The polynomial $f$ involves only $b_2,b_4,b_6$, and these first appear at $p=3,5,7$ respectively, so each prime in turn contributes exactly one new invariant.
At $p=11$ and beyond, \Cref{lem.charsum} returns a relation among $b_2,b_4,b_6$ that is already known; consistently, $b_8$ is not independent, since $4b_8=b_2b_6-b_4^2$.
It also explains the one asymmetry in Theorem~\ref{thm.main}: the $p=7$ step needs $b_4$ modulo $7$, whereas the $p=5$ step delivers it only modulo $5$, which is why \eqref{eq.w6} requires the integer $w_4$ and not merely its residue modulo $5$.
Finally, Siksek closes \cite{siksek} by asking for a more conceptual understanding of the congruence conditions satisfied by the coefficients of isogenous curves.
For the conditions of \cite[Theorem~1.1]{siksek}, recalled in the introduction, Theorem~\ref{thm.main} provides one: isogenous curves $E$ and $E'$ share their Dirichlet coefficients and conductor, so the formulae force $w_1$, $w_2$, $w_3$ and $w_4\bmod 5$ to agree, and \eqref{eq.w6} converts $w_4'=w_4-5t$ into $w_6'\equiv w_6-15b_2t\equiv w_6-b_2t\pmod 7$.
\end{rem}

\begin{rem}\label{rem.vijay}
Vijay Srinivasan has communicated to us two further perspectives on these congruences.
The first organises the count in Theorem~\ref{lem.w1a2} by the involution $(x,y)\mapsto(x,\,-y-\tilde{w}_1x-\tilde{w}_3)$ of $\tilde{E}$: points not fixed by it cancel in pairs modulo $2$, while its fixed points --- the point at infinity, together with the single point of $\tilde{E}$ above each zero of $\tilde{w}_1x+\tilde{w}_3$ --- number $1+\tilde{w}_1$ modulo $2$.

The second unifies Theorem~\ref{lem.w1a2} with \Cref{lem.charsum}.
The special fibre $\tilde{E}/\F_p$ is a plane cubic, so $H^0(\tilde{E},\mathcal{O}_{\tilde{E}})$ and $H^1(\tilde{E},\mathcal{O}_{\tilde{E}})$ are one-dimensional whatever the reduction type, and Fulton's fixed point formula for varieties over finite fields gives
\[\#\tilde{E}(\F_p)\equiv 1-\lambda \ \pmod p,\]
where $\lambda\in\F_p$ is the scalar through which the $p$-power Frobenius acts on $H^1(\tilde{E},\mathcal{O}_{\tilde{E}})$; equivalently, $a_p(E)\equiv\lambda \pmod p$ by \Cref{lem.a2count}.
A \v{C}ech computation --- the classical calculation of the Hasse invariant --- identifies $\lambda$ as $\big[x^{p-1}\big]f(x)^{(p-1)/2}$ for odd $p$, and as $\tilde{w}_1$ for $p=2$.
In this light, the formulae of Theorem~\ref{thm.main} invert, prime by prime, the trace of Frobenius on the coherent cohomology of the special fibre.
\end{rem}

\begin{rem}\label{rem.sec4_ai}
    We first obtained the expressions for $w_1,w_2,w_3$ by a brute-force approach, guided by the observed classification of decision trees. During a discussion with ChatGPT--5.6 Sol, it was suggested that $w_4 \pmod 5$ could likewise be recovered from $a_5$, and $w_6 \pmod 7$ from $a_7$. In working out these constructions, we realized that an analogous argument also yields the expression for $w_2$, and that the proofs for $w_1$ and $w_3$ could be reformulated along the same lines. The proofs themselves were written by us, inspired by these conversations. 
\end{rem}

\subsection{Geometric interpretation}\label{s.geom}
We may interpret this result geometrically as follows. By the \emph{local reduction data} of $E$ at a prime $p$ we mean the reduction type of $E$ at $p$, refined by the split/nonsplit distinction in the multiplicative case and by the ordinary/supersingular distinction in the good case.

\begin{cor}\label{cor.w1redn}
If $E$ is an elliptic curve over $\mathbb{Q}$ in reduced minimal Weierstrass form, then
\[w_1=\begin{cases}0,&\text{ reduction at }2\text{ is additive or good supersingular},\\
1,&\text{ reduction at }2\text{ is multiplicative or good ordinary}.\end{cases}\]
\end{cor}

\begin{cor}\label{cor.w2redn}
Let $E$ be an elliptic curve over $\mathbb{Q}$ written as a reduced minimal Weierstrass equation.
\begin{enumerate}
\item If $E$ has additive, multiplicative, or good supersingular reduction at $3$, then $w_2$ is determined by the local reduction data of $E$ at $2$ and $3$.
\item If $E$ has good ordinary reduction at $3$, then $w_2$ is determined by the local reduction data of $E$ at $2$ together with $a_3(E)$.
\end{enumerate}
\end{cor}

\begin{proof}
By equation~\eqref{eq.w2a3w1}, $w_2$ is determined by $w_1$ together with the residue $[a_3(E)]_3$, and $w_1$ is in turn determined by the local reduction data at $2$ (Corollary~\ref{cor.w1redn}). Since $a_3(E)$ equals $0$ for additive reduction and $1$ (resp.\ $-1$) for split (resp.\ nonsplit) multiplicative reduction, while good reduction at $3$ is ordinary precisely when $3\nmid a_3(E)$, the residue $[a_3(E)]_3$ equals $0$, $1$, $2$, $0$ in the additive, split multiplicative, nonsplit multiplicative and good supersingular cases at $3$, respectively, and is given by $a_3(E)$ itself in the good ordinary case.
\end{proof}

We note that the value of $a_p(E)$ mod $p$ determines which of the two Frobenius eigenvalues is the $p$-adic unit.

\begin{cor}\label{c.w3}
Let $E$ be an elliptic curve over $\mathbb{Q}$ written as a reduced minimal Weierstrass equation.
\begin{enumerate}
\item If $E$ has additive or good supersingular reduction at $2$, then $w_3$
is determined by the reduction type at $2$ alone; explicitly, $w_3=0$ in the
additive case and $w_3=1$ in the good supersingular case.
\item If $E$ has multiplicative reduction at $2$, then $w_3$ is determined by
the local reduction data at $2$ together with: the local reduction data at
$3$, if $E$ does not have good ordinary reduction at $3$; or $a_3(E)$, if it
does.
\item If $E$ has good ordinary reduction at $2$, then $w_3$ is determined by
$a_2(E)$ together with: the local reduction data at $3$, if $E$ does not have
good ordinary reduction at $3$; or $a_3(E)$, if it does.
\end{enumerate}
\end{cor}

\begin{proof}
Recall that $a_p(E)$ equals $0$ for additive reduction and $1$ (resp.\ $-1$)
for split (resp.\ nonsplit) multiplicative reduction, and that good reduction
at $p$ is ordinary precisely when $p\nmid a_p(E)$.

(1) Here $a_2(E)$ is even, so equation~\eqref{eq.w3} gives $w_3=[N(E)]_2$,
which is $0$ in the additive case ($2\mid N(E)$) and $1$ in the good
supersingular case ($2\nmid N(E)$).

(2) and (3): Here $a_2(E)$ is odd, so $a_2(E)\in\{-1,1\}$ by the Hasse bound,
and $w_1=1$ by Theorem~\ref{lem.w1a2}; equation~\eqref{eq.w3} thus reads
\[
w_3=\Big[N(E)+1+w_2+\tfrac{a_2(E)+1}{2}\Big]_2 ,
\]
and it remains to account for the three inputs on the right. First,
$[N(E)]_2$ equals $0$ in case (2) and $1$ in case (3). Second, $a_2(E)$ is
the split/nonsplit datum at $2$ in case (2), and is given in case (3).
Third, since $w_2\in\{-1,0,1\}$, the coefficient $w_2$ is even if and only if
$w_2=0$, which by equation~\eqref{eq.w2a3w1} amounts to
$a_3(E)\equiv1\pmod{3}$; whether this holds is determined by the local
reduction data at $3$ when the reduction there is additive, multiplicative or
good supersingular ($a_3(E)$ being $0$, $\pm1$, or divisible by $3$,
respectively), and by $a_3(E)$ itself when it is good ordinary.
\end{proof}

\begin{rem}
We present a somewhat less verbose version of Corollary~\ref{c.w3}.
\begin{itemize}
\item In the case that $E$ has bad reduction at $2$:
\begin{itemize}
\item If $a_2(E)$ is even, then $w_3=0$.
\item If $a_2(E)$ is odd, then
\[w_3=\begin{cases}1-[w_2]_2,&a_2(E)=-1,\\ {[w_2]_2},&a_2(E)=1.\end{cases}\]
\end{itemize}
\item In the case that $E$ has good reduction at $2$:
\begin{itemize}
\item If $a_2(E)$ is even, then $w_3=1$.
\item If $a_2(E)$ is odd, then
\[w_3=\begin{cases}[w_2]_2,&a_2(E)=-1,\\ 1-[w_2]_2,&a_2(E)=1.\end{cases}\]
\end{itemize}
\end{itemize}
\end{rem}

\section{Applications to computing tables of elliptic curves}\label{s.ectables}

We close with a suggestion of John Cremona for applying Theorem~\ref{thm.main} to the computation of tables of elliptic curves.

Tables of elliptic curves over $\Q$ of given conductor $N$, such as those of \cite{cremona1997algorithms} and \cite{lmfdb}, are computed from the rational newforms $f$ of weight $2$ for $\Gamma_0(N)$.
Let $E_f$ be the curve attached to $f$, so that $a_n(f)=a_n(E_f)$ in the notation of the introduction.
Using modular symbols, one computes floating point approximations to a basis of the period lattice of $f$, and from these, approximations to the invariants $c_4$ and $c_6$ of $E_f$.
These are rounded to integers, and the reduced model with these invariants is recovered as follows \cite[\S2.14, \S3.1]{cremona1997algorithms}: $b_2$ is the residue of $-c_6$ modulo $12$ lying in $\{-5,\dots,6\}$, then
\begin{equation}\label{eq.c4c6tob}
b_4=\frac{b_2^2-c_4}{24},\qquad b_6=\frac{-b_2^3+36b_2b_4-c_6}{216},
\end{equation}
and finally
\begin{equation}\label{eq.btow}
w_1=[b_2]_2,\quad w_3=[b_6]_2,\quad w_2=\frac{b_2-w_1}{4},\quad w_4=\frac{b_4-w_1w_3}{2},\quad w_6=\frac{b_6-w_3}{4}.
\end{equation}
The model is then checked to be minimal and to have conductor $N$, and its Frobenius traces are compared with the $a_p(f)$.
The periods are evaluated as sums over the Fourier coefficients $a_n(f)$, and the number of coefficients required grows with the precision to which $c_4$ and $c_6$ must be known for the rounding to be correct \cite[\S\S2.10--2.14]{cremona1997algorithms}.
In double precision one can typically recognise $c_4$, but the last digit or two of $c_6$ may be undetermined, and congruence conditions on $c_4$ and $c_6$ are then used to fix them \cite[\S2.14]{cremona1997algorithms}.

\begin{rem}\label{rem.cremona}
Cremona has communicated to us that Theorem~\ref{thm.main} permits the following variant.
The coefficients $a_2(f)$ and $a_3(f)$ and the level $N$ are known before any period is computed, so \eqref{eq.thmw1}--\eqref{eq.thmw3}, with $\varepsilon=[N]_2$, give $w_1$, $w_2$, $w_3$, and hence $b_2$, exactly.
In particular $w_1$, $w_2$ and $w_3$, which \eqref{eq.btow} obtains from the exact values of $c_4$ and $c_6$, are found without them.
Substituting the approximations to $c_4$ and $c_6$ into \eqref{eq.c4c6tob} and \eqref{eq.btow} then gives approximations
\begin{equation}\label{eq.w4w6approx}
w_4=\frac{b_2^2-c_4}{48}-\frac{w_1w_3}{2},\qquad
w_6=\frac{-b_2^3+36b_2b_4-c_6}{864}-\frac{w_3}{4},
\end{equation}
in which $b_4=w_1w_3+2w_4$ is exact once $w_4$ has been rounded, and one rounds $w_4$ and $w_6$ in place of $c_4$ and $c_6$.
Rounding $w_4$ correctly requires $c_4$ to within $24$, and rounding $w_6$ then requires $c_6$ to within $432$, whereas rounding $c_4$ and $c_6$ themselves requires each to within $1/2$.
The tolerated errors thus increase by factors of $48$ and $864$; the latter is worth almost three decimal digits of $c_6$.
The periods are therefore needed to less precision, and fewer Fourier coefficients enter their evaluation.
The congruences \eqref{eq.thmw4} and \eqref{eq.thmw6} serve the same purpose as the congruence conditions on $c_4$ and $c_6$: rounding $w_4$ to the nearest integer in the residue class modulo $5$ given by $a_5(f)$, and then $w_6$ to the nearest integer in the residue class modulo $7$ given by $a_7(f)$ and $w_4$, tolerates errors in $c_4$ and $c_6$ up to $120$ and $3024$ respectively.
The resulting model is verified as before.
Tingley's thesis \cite{tingley}, in which the first computations of modular elliptic curves were made, likewise found $w_1$, $w_2$, $w_3$ first and then computed $w_4$ and $w_6$ directly rather than through $c_4$ and $c_6$; this was not implemented in \cite{cremona1997algorithms}.
The saving is modest: once $f$ has been isolated, its Fourier coefficients are cheap to compute, and the cost of the whole computation is dominated by the linear algebra on the space of modular symbols, whose dimension grows with $N$.
\end{rem}

\appendix

% sections/unused/appendix-tables.tex holds the old brute-force tables modulo 2 and 3;
% they are superseded by the conceptual proofs in Section~\ref{s.proofs} and no longer input.
\section{Transformers}\label{a.transformers}
A decision tree is a classical, interpretable model based on recursive partitioning.
On the other hand, a transformer is a deep learning architecture that captures complex patterns via attention mechanisms. 
In this section, we use a dataset $\mathcal{E}_2$ consisting of all isogeny classes of conductor $\leq5\times10^5$ from \cite{lmfdb}.
This dataset contains approximately $2.1$ million isogeny classes.
We implement an encoder--decoder transformer \cite{vaswani2017attention}, closely following the \texttt{Int2Int} framework \cite{int2int}; complete code and configuration files are available at \cite{H26}.

\subsection{Architecture and data representation}
The model and training set-up are deliberately minimal. It implements an encoder--decoder transformer with a single encoder layer, a single decoder layer, and one attention head. The model dimension is $d_{\mathrm{model}}=128$ and the feed-forward dimension is $256$, with no dropout. Positional information is supplied by learnable positional embeddings.

Each integer is written in base $b=30$ and encoded as a sign token ($+$ or $-$) followed by its base-$30$ digits; together with the special tokens $\langle\textsc{bos}\rangle,\langle\textsc{eos}\rangle,\langle\textsc{pad}\rangle,\langle\textsc{sep}\rangle,\langle\textsc{unk}\rangle$ this yields a vocabulary of size $37$. A sequence of integers is tokenised by concatenating these encodings between a leading $\langle\textsc{bos}\rangle$ and a trailing $\langle\textsc{eos}\rangle$.

The source sequence is the vector of Frobenius traces $(a_p(E))_{p<100}$, i.e.\ the $25$ traces $a_2(E),a_3(E),\dots,a_{97} (E)$. For the experiments predicting $w_3$ we additionally append the conductor $N(E)$ to the source sequence, motivated by equation~\eqref{eq.w3} and Figure~\ref{fig:dt_nmod2_a2_w2}, in which $w_3$ depends on $N(E) \bmod 2$. The target sequence is the tuple of reduced minimal Weierstrass coefficients to be predicted, e.g.\ $(w_1,w_2,w_3)$.

We optimise the token-level cross-entropy of the shifted decoder output using Adam (learning rate $10^{-4}$, weight decay $10^{-2}$) with batch size $32$, splitting $\mathcal{E}_2$ into $80\%$ training and $20\%$ validation data. 
We report the proportion of validation classes for which the \emph{entire} predicted tuple of integers is correct.

\subsection{Experimental results}
We first produce, in the transformer setting, the picture obtained from decision trees in \Cref{s.experiments}.
Predicting the map
\[(a_p(E))_{p<100}\longmapsto(w_1,w_2)\]
the model attains $100\%$ validation accuracy, in agreement with the fact that $(w_1,w_2)$ is determined by $(a_2(E),a_3(E))$ (Theorems~\ref{lem.w1a2} and~\ref{lem.formulaw2}).

We next consider the full map
\[\big((a_p(E))_{p<100},\,N(E) \big)\longmapsto(w_1,w_2,w_3),\]
which is the transformer analogue of the $w_3$ experiments of Section~\ref{s.w3}.
Table~\ref{tab:transformer_w123} reports the per-coefficient validation accuracy, together with the majority-class baseline for each coefficient.
The coefficient $w_1$ is recovered exactly, consistent with $w_1\equiv a_2(E)\bmod 2$ (Theorem~\ref{lem.w1a2}); $w_2$ is recovered with high but imperfect accuracy; and $w_3$ is the clear bottleneck, with the joint accuracy on $(w_1,w_2,w_3)$ governed almost entirely by the errors in $w_3$.

\begin{table}[h]
    \centering
    \begin{tabular}{c|c|c|c|c}
    \toprule
     & $w_1$ & $w_2$ & $w_3$ & joint $(w_1,w_2,w_3)$ \\
    \midrule
    transformer & 100\% & 89.2\% & 80.5\% & 72.6\% \\
    \bottomrule
    \end{tabular}
    \caption{Per-coefficient and joint validation accuracy of the transformer predicting $(w_1,w_2,w_3)$ from $\big((a_p(E))_{p<100},N(E) \big)$ on $\mathcal{E}_2$.}
    \label{tab:transformer_w123}
\end{table}

The $w_3$ accuracy of $80.5\%$ exceeds the best decision-tree accuracy of $76.68\%$ obtained from $(a_p(E))_{p<10^4}$ alone (Table~\ref{tab:dtw3}); we attribute the improvement to the inclusion of the conductor $N(E)$ in the input, in line with the appearance of $N(E) \bmod 2$ in equation~\eqref{eq.w3}.
Nevertheless, the transformer falls well short of the $100\%$ accuracy that the closed-form expression of Theorem~\ref{thm.w3} guarantees once the correct features $(a_2(E),w_2,N(E) \bmod 2)$ are available, indicating that the model has not, on its own, discovered the exact arithmetic rule for $w_3$.

\bibliographystyle{amsalpha}
\bibliography{references}

\end{document}